\documentclass[letterpaper, 10 pt, conference]{ieeeconf} 
\IEEEoverridecommandlockouts                              

\usepackage{graphicx}
\usepackage{amsmath}
\usepackage{dsfont}
\usepackage{xcolor}
\usepackage{amsfonts}
\usepackage{amssymb}
\usepackage{epsfig,psfrag,pstricks}
\usepackage{tikz}
\usepackage{tkz-berge}
\usetikzlibrary{shapes,arrows}
\DeclareMathOperator{\im}{im}

\DeclareMathOperator{\spa}{span}
\DeclareMathOperator{\sat}{sat}

\newtheorem{theorem}{Theorem}

\newtheorem{definition}[theorem]{Definition}
\newtheorem{proposition}[theorem]{Proposition}
\newtheorem{corollary}[theorem]{Corollary}
\newtheorem{exmp}{Example}[section]
\newtheorem{remark}[theorem]{Remark}

\title{\LARGE \bf
Constrained proportional integral control of dynamical distribution
networks with state constraints*
}

\author{Jieqiang Wei$^{1}$ and Arjan van der Schaft$^{2}$
\thanks{*The work of the first author is supported by the Chinese Science
Council (CSC). The research of the second author leading to these
results has received support from the EU 7th Framework Programme
[FP7/2007-2013] under grant agreement no. 257462 HYCON2 Network of
Excellence.}
\thanks{$^{1}$,$^{2}$Johann Bernoulli Institute for Mathematics and Computer
Science, University of Groningen, PO Box 407, 9700 AK, the
Netherlands.\newline {\tt\small J. Wei@rug.nl, A.J.van.der.Schaft@rug.nl}}%
}

\begin{document}

\maketitle
\thispagestyle{empty}
\pagestyle{empty}

\begin{abstract}
This paper studies a basic model of a dynamical distribution network,
where the network topology is given by a directed graph with storage
variables corresponding to the vertices and flow inputs
corresponding to the edges. We aim at regulating the system to consensus, while
the storage variables remain greater or equal than a given lower bound. 
The problem is solved by using a distributed PI controller structure with 
constraints which vary in time. It is shown how the constraints can be obtained
by solving an optimization problem. 
\end{abstract}

\section{INTRODUCTION}

In this paper we continue our study of the dynamics of distribution
networks. Identifying the network with a directed graph we associate with every
vertex of the graph a state variable corresponding to {\it storage}, and with
every edge a control input variable corresponding to {\it flow}, possibly 
subject 
to constraints. In previous work \cite{Wei2012,Wei2013,Weimtns2014} it has
been shown under which conditions a constrained proportional-integral (PI)
controller will regulate the system towards output agreement in the presence of 
unknown constant external disturbances (corresponding to constant in/outflows of 
the network). 

In many cases of practical interest it is
natural to require that the state variables of the distribution network will remain larger than a given
minimal value, e.g. zero. A hydraulic network with state variables being the
storage of fluid is a clear example of such a situation. On the other hand, the
previously developed PI-controller can give rise to damped oscillatory behavior
which may violate such state constraints. The aim of the current paper is to
modify the PI-controller in such a way that the lower bounds for the
state variables will be satisfied for all time while the system will
still converge to output agreement. This is done by adapting the constraints of 
the PI controller. 

The main related work can be summarized as follows.. In \cite{depersis} an alternative 
scheme is given in order that the state variables remain nonnegative. However, 
this scheme does not respect mass conservation. In \cite{Blanchini2000}, the
authors consider a similar distribution network model but with a proportional 
controller instead of a PI controller. A discontinuous Lyapunov-based 
controller is given to stabilize the system without violating the storage and 
flow
constraints; however it is not robust 
with respect to constant disturbances). In \cite{Bauso2011}, 
using the same model as in \cite{Blanchini2000}, the authors focus on a
different problem of driving the state to a small neighborhood of the reference
value and relate the control input value at equilibrium to an optimization
problem. 

The structure of the paper is as follows. Preliminaries and notations are given 
in Section II. In Section \ref{model}, we introduce the basic model which can be 
identified as a port-Hamiltonian system, in line with the
general definition of port-Hamiltonian system on graphs
\cite{schaftSIAM, schaftCDC08, schaftBosgrabook,
schaftNECSYS10}; see also \cite{allgower11,Mesbahi11}. In 
Section \ref{review} we briefly recall from
our previous work \cite{Wei2012,Wei2013,Weimtns2014} the necessary and 
sufficient graphical conditions in order that a constrained PI controller 
structure will solve the regulation problem.

In Section \ref{problem} we formulate the main problem of this paper, namely the
adaptation 
of the constraints of the PI-controller such that the system will reach
output agreement while the state variables will remain greater or equal than a 
given
lower bound. In Section \ref{design} an optimal
control protocol for the adaptation of the flow (control) constraints is 
developed, while stability analysis of the scheme is given in
Section \ref{proof}. The conclusions are in Section \ref{conclusion}.

\section{Preliminaries and notations}
We recall some standard definitions regarding directed graphs,
as can be found e.g. in \cite{Bollobas98}. A \textit{directed graph}
$\mathcal{G}$ consists of a finite set $\mathcal{V}$ of \textit{vertices}
and a finite set $\mathcal{E}$ of \textit{edges}, together
with a mapping from $\mathcal{E}$ to the set of ordered pairs of
$\mathcal{V}$, where no self-loops are allowed. Thus to any edge
$e\in\mathcal{E}$ there corresponds an ordered pair
$(v,w)\in\mathcal{V}\times\mathcal{V}$
(with $v\not=w$), representing the tail vertex $v$ and the head
vertex $w$ of this edge.

A directed graph is specified by its \textit{incidence
matrix} $B$, which is an $n\times m$ matrix, $n$ and $m$ being the
number of vertices and edges respectively, with $(i,j)^{\text{th}}$
element equal to $1$ if the $j^{\text{th}}$ edge is towards vertex
$i$, and equal to $-1$ if the $j^{\text{th}}$ edge is originating from
vertex $i$, and $0$ otherwise. In this paper `graph' will throughout mean 
`directed graph' unless stated otherwise.
A graph is {\it strongly connected} if it is
possible to reach any vertex starting from any other vertex by traversing edges
following their directions. It is called {\it weakly connected}
if it is possible to reach any vertex from every other vertex using the edges
{\it not} taking into account their direction. A graph is weakly connected if
and only if $\ker B^T = \spa \mathds{1}_n$. Here $\mathds{1}_n$ denotes the
$n$-dimensional vector with all elements equal to $1$. A graph that is not
weakly connected falls apart into a number of weakly connected components. The 
number of connected components is equal to $\dim
\ker B^T$.
For each vertex, the number of incoming and outgoing edges are called the {\it 
in} and {\it out-degree} of the vertex respectively. A graph is called
{\it balanced} if and only if the in-degree and out-degree of every vertex are
equal to each other. A graph is balanced if and only if $\mathds{1}_n \in \ker B$.

Given a graph, we define its \textit{vertex space} as the vector space of all
functions from $\mathcal{V}$ to some linear space $\mathcal{R}$. In the rest of
this paper we will take $\mathcal{R}=\mathbb{R}$, in which case
the vertex space can be identified with $\mathbb{R}^{n}$. Similarly, we define
its \textit{edge space} as the
vector space of all functions from $\mathcal{E}$ to $\mathcal{R} = \mathbb{R}$,
which can be identified with $\mathbb{R}^{m}$. In this way, the incidence matrix
$B$ of the graph can be also regarded as the matrix representation of a linear
map from the edge space $\mathbb{R}^m$ to the vertex space $\mathbb{R}^n$. Extensions of the results to general linear spaces $\mathcal{R}$ are straightforward.

\noindent
{\bf Notation}: For $a,b\in\{\mathbb{R},\pm \infty\}^m$ the notation $a \leqslant b$ (resp. $ a
< b$) will denote element-wise inequality $a_i \leq b_i$ (resp. $a_i <
b_i$),\,$i=1,\ldots,m$. For $a <
b$ the multidimensional
saturation function
$\sat(x\,;a,b): \mathbb{R}^m\rightarrow\mathbb{R}^m$ is defined as
\begin{equation}
\sat(x\,;a,b)_i  = \left\{ \begin{array}{ll}
a_i & \textrm{if $x_i\leq a_i,$}\\
x_i & \textrm{if $a_i<x_i<b_i,$}\\
b_i & \textrm{if $x_i\geq b_i$},
\end{array}
\, i=1,\ldots,m. \right.
\end{equation}

\section{A dynamical network model with PI controller}\label{model}

Let us consider the following dynamical system defined on the vertices of graph
(\cite{schaftNECSYS10,schaftCDC08,schaftSIAM}){\small
\begin{equation}\label{plant}
\begin{array}{rcl}
\dot{x} & = & u + E\bar{d}, \quad x,u \in \mathbb{R}^n, \quad d\in
\mathbb{R}^k \\[2mm]
y & = & \frac{\partial H}{\partial x}(x), \quad y \in \mathbb{R}^n,
\end{array}
\end{equation}
}where $H: \mathbb{R}^n \to \mathbb{R}$ is a differentiable function, and
$\frac{\partial H}{\partial x}(x)$ denotes the column vector of partial
derivatives of $H$. Here $x_i$ and $u_i$ are the state
and input variable associated to the $i^{\text{th}}$ vertex of the graph
respectively. And $E$ is an $n \times k$ matrix whose columns consist of exactly
one entry equal to $1$ (inflow) or $-1$ (outflow), while the rest of the
elements is zero. Thus $E$ specifies the $k$ terminal vertices where flows can
enter or leave the network (\cite{vanderschaftmaschkearchive}). Finally, $\bar{d}$ is a
vector of constant disturbances. System (\ref{plant}) defines a port-Hamiltonian
system
(\cite{vanderschaftmaschkearchive,
vanderschaftbook}), satisfying the energy-balance {\small
\begin{equation}
\frac{d}{dt}H = u^Ty+\frac{\partial^T H}{\partial x}(x)E\bar{d}.
\end{equation}}

Here the state variables on vertices are controlled by the flow in the edges of
network in the following manner {\small
\begin{equation}\label{F-connection}
 u=B\mu, \quad \mu\in\mathbb{R}^m
\end{equation}
}where $\mu_j$ is a flow variable
associated to the $j^{\text{th}}$ edge of the graph. In this paper we consider
the case when the controller is defined on the edges to provide the flow
variables. As explained in \cite{Wei2013}, when $d\neq0$, the proportional
control will not be sufficient to reach load balancing. Hence we consider a
proportional-integral (PI) controller given by the dynamic output feedback 
{\small \begin{equation}\label{PI}
 \begin{array}{rcl}
  \dot{\eta} & = & \zeta, \quad \eta, \zeta\in\mathbb{R}^m, \\ [2mm]
\mu & = & -R\zeta-\frac{\partial H_c}{\partial \eta}(\eta)
 \end{array}
\end{equation}
}where $\eta_i$ is state variable associated to
$i^{\text{th}}$ edge, $R$ is a diagonal matrix with strictly positive diagonal
elements $r_1,r_2,\ldots,r_m$, and $H_c$
the Hamiltonian function corresponding to the controller. Here the controller
is driven by the relative output of the systems (\ref{plant}) on vertices, i.e., {\small
\begin{equation}\label{P-connection}
 \zeta=B^Ty
\end{equation}}

The closed-loop system of 
(\ref{plant},\ref{F-connection},\ref{PI},\ref{P-connection}) is a 
port-Hamiltonian system {\small
\begin{equation}\label{closedloop}
\begin{bmatrix} \dot{x} \\[2mm] \dot{\eta} \end{bmatrix} =
\begin{bmatrix} -BRB^T & -B \\[2mm] B^T & 0 \end{bmatrix}
\begin{bmatrix} \frac{\partial H}{\partial x}(x) \\[2mm] \frac{\partial
H_c}{\partial \eta}(\eta) \end{bmatrix} +
\begin{bmatrix} E \\[2mm] 0 \end{bmatrix} \bar{d},
\end{equation}
}with total Hamiltonian {\small
\begin{equation}
H_{\mathrm{tot}}(x,\eta)\\ := H(x) + H_c(\eta).\nonumber 
\end{equation}}

Suppose now that the constant disturbance $\bar{d}$ satisfies the
{\it matching condition}, i.e., there exists a 
controller state $\bar{\eta}$ such that{\small
\begin{equation}\label{matching}
E \bar{d} = B\frac{\partial H_c}{\partial \eta}(\bar{\eta}).
\end{equation}
}By modifying the total Hamiltonian $H_{tot}$ into{\small 
\begin{equation}
 V_{\bar{d}}(x,\eta): = H(x)+H_c(\eta)-\frac{\partial^T H_c}{\partial
\eta}(\bar{\eta})(\eta-\bar{\eta})-H_c(\bar{\eta})
\end{equation}
}which serves as a candidate Lyapunov function, we can obtaine
the following theorem.
\begin{theorem}(\cite{Wei2012, Wei2013})
Consider the dynamical system 
(\ref{plant},\ref{PI},\ref{F-connection},\ref{P-connection}) on the graph 
$\mathcal{G}$. Let the constant disturbance $\bar{d}$ satisfies the matching 
condition (\ref{matching}) with a $\bar{\eta}$. Assume $V_{\bar{d}}(x,\eta)$ is 
radially unbounded. Then the trajectories of the
closed-loop system (\ref{closedloop}) will converge to an element of the load
balancing set{\small
\begin{equation}
 \mathcal{E}_{\mathrm{tot}} = \{ (x,\eta) \mid \frac{\partial H}{\partial x}(x)=
 \alpha \mathds{1}, \, \alpha \in \mathbb{R}, \, B\frac{\partial H_c}{\partial
 \eta}(\eta) = E\bar{d}\, \}. \nonumber
\end{equation}
}if and only if $\mathcal{G}$ is weakly connected.
\end{theorem}
\begin{corollary}
If $\ker B = 0$, which is equivalent \cite{Bollobas98} to the graph having no
{\it cycles}, then for every $\bar{d}$ there exists a {\it unique} 
$\frac{\partial H_c}{\partial \eta}(\bar{\eta})$
satisfying (\ref{matching}). Suppose $H_c$ has positive definite Hessian 
matrix,  then in (\ref{matching}) $\bar{\eta}$ is also {\it 
unique} and the convergence is towards the set
$\mathcal{E}_{\mathrm{tot}} = \{ (x, \bar{\eta})
\mid \frac{\partial H}{\partial x}(x) = \alpha \mathds{1}, \alpha \in
\mathbb{R}\}$.
\end{corollary}
\begin{corollary}
In case of the standard quadratic Hamiltonians $H(x) = \frac{1}{2} \| x \|^2,
H_c(\eta)=\frac{1}{2} \| \eta \|^2$ there exists for every $\bar{d}$
a controller state $\bar{\eta}$ such that (\ref{matching}) holds if and only 
if{\small
\begin{equation}\label{matching1}
\im E \subset \im B.
\end{equation}
}Furthermore, in this case $V_{\bar{d}}$ equals the radially unbounded function
$\frac{1}{2}
\| x \|^2 + \frac{1}{2} \| \eta - \bar{\eta} \|^2$, while convergence will be
towards the load balancing set $\mathcal{E}_{\mathrm{tot}} = \{ (x,\eta) \mid x =
\alpha
\mathds{1}, \alpha \in \mathbb{R},\, B\eta = E\bar{d}\}$.
\end{corollary}

A necessary (and in case the graph is weakly connected necessary {\it and}
sufficient) condition for the inclusion $\im E \subset \im B $ is that
$\mathds{1}^TE =
0$. In its turn $\mathds{1}^TE =
0$ is equivalent to the fact that for every $\bar{d}$ the total inflow into the
network equals to the
total outflow). The condition $\mathds{1}^TE = 0$ also implies{\small
\begin{equation}
\mathds{1}^T\dot{x} = -\mathds{1}^TB\mu +
\mathds{1}^TE\bar{d}=0, \nonumber
\end{equation}
}implying (as in the case $d=0$) that $\mathds{1}^Tx$ is a {\it conserved
quantity} for the closed-loop
system (\ref{closedloop}).

\section{Results review for the case with flow constraints}\label{review}
In many cases of interest, the flows in the edges are constrained. Here we 
briefly review the main results in \cite{Wei2013,Weimtns2014} where we consider instead of (\ref{F-connection}) its constrained version {\small
\begin{equation}\label{F-connection-constrained}
 u=B\sat(\mu,\mu^-,\mu^+)
\end{equation}
}where $\mu^-, \mu^+\in\mathbb{R}^m$ 
are the lower and upper bounds for the flow constraints. 
For simplcity of exposition we only consider the PI controller (\ref{PI}) 
with $H_c(\cdot)=\frac{1}{2}\|\cdot\|_2^2$ and $R=I$.

As is shown in \cite{Wei2013,Weimtns2014}, without loss of generality we can assume
that the flow constraints satisfy $u^+\geq u^-\geq 0$ and the two equality signs do not hold at the same time. We call this kind of the constraints are compatible with the orientation. Furthermore, adding the disturbance 
satisfying the constrained version of the matching condition, i.e.,{\small
\begin{equation}\label{constrained-matching}
E \bar{d} = B\sat(\bar{\eta}; \mu^-,\mu^+),
\end{equation}
}is equivalent to translation of the constraints. It follows 
that, without loss of generality, we can focus on the closed-loop system without
disturbance{\small
\begin{equation}\label{constant_constraint_input}
\begin{array}{rcl}
\dot{x} & = & B \sat(-B^T\frac{\partial H}{\partial
x}(x)-\eta, \mu^-, \mu^+), \\[2mm]
\dot{\eta} & = & B^T\frac{\partial H}{\partial x}(x).
\end{array}
\end{equation}
}The main results about system (\ref{constant_constraint_input}) are summarized in
\begin{theorem}(\cite{Wei2013})
Consider the dynamical system
(\ref{constant_constraint_input}) with compatible flow constraints. Then for
any $\mu^- < \mu^+\in\mathbb{R}_+^m$ such that
$\cap_{i=1}^{m}[\mu^-_i,\mu^+_i]$ contains an open interval,
the trajectories of (\ref{constant_constraint_input}) converge to{\small
\begin{equation}
\begin{aligned}
 \mathcal{E}_{\mathrm{tot}} = &\{ (x,\eta) \mid \frac{\partial H}{\partial x}(x)
=\alpha \mathds{1}, \, \alpha \in \mathbb{R}, \\
& B\sat(-\eta, \mu^-,\mu^+)=0
\, \}
\end{aligned}
\end{equation}
}if and only if the graph is {\it strongly connected} and {\it
balanced}.
\end{theorem}
 
For any network with given orientation and constraints on the edges, we can
define the \textit{interior point condition}.

\begin{definition}\label{interior point}
 (Interior Point Condition) Given a directed graph with arbitrary constraints
$[\mu^-, \mu^+]$, the network will be
said to satisfy the interior point condition if there exists a vector
$z\in[\mu^-, \mu^+]\cap \ker B$ such that the subgraph 
$\mathcal{G}_0=\{\mathcal{V},\mathcal{E}_0\}$ is weakly connected where 
{\small
\begin{equation}
\mathcal{E}_0(z; \mu^-, \mu^+) = \{ e_i \mid e_i\in\mathcal{E}, z_i\in 
(\mu^-_i,\mu^+_i)\}.\nonumber
\end{equation}
}

\end{definition}
\begin{theorem}\label{main}
 (\cite{Weimtns2014}) Consider the dynamical system
(\ref{constant_constraint_input}) defined on a
weakly connected graph. Then the trajectories will converge
to {\small
\begin{equation}
 \mathcal{E}_{\mathrm{tot}} = \{ (x,\eta) \mid \frac{\partial H}{\partial x}(x) =
\alpha \mathds{1}_n, \, B\sat(-\eta\,;u^-,u^+) = 0 \}.\nonumber
\end{equation}
}if and only if the network satisfies the interior point condition.
\end{theorem}

\section{Constrained PI-controllers maintaing a lower bound for the state 
variables}\label{problem}

Although, as summarized in the previous sections, the PI controller with both 
unconstrained (\ref{F-connection}) and constrained flow connection 
(\ref{F-connection-constrained}) is successful in obtaining output agreement 
for the plant, it introduces oscillatory behavior which may cause the state 
variables $x$ become smaller than some given lower bounds. For certain
applications this may be undesirable or infeasible, as is illustrated by the
following example.

\begin{exmp}[Hydraulic network]\label{Hydraulic network}
Consider a hydraulic network, modeled as a directed graph with vertices (nodes)
corresponding to reservoirs, and edges (branches) corresponding to pipes.
Let $x_i$ be the volume of fluid stored at vertex $i$, and $\mu_j$
the flow through edge $j$. Then the mass balance of the network is
summarized as (\ref{plant}) and (\ref{F-connection}). Let $H(x)$ denote
the stored energy in the reservoirs (e.g., gravitational energy). For
cylindric reservoirs, $x_i=S_ih_i, H_i=\frac{1}{2}\rho S_igh_i^2=\frac{\rho
g}{2S_i}x_i^2$ where $S_i$ is the bottom area, $h_i$ is the
height of liquid of $i$th reservoir respectively, and $g$ is gravity
coefficient. Then $P_i :=
\frac{\partial H}{\partial x_i}(x)=\rho g h_i=\frac{\rho gx_i}{S_i}, i=1,
\ldots, n,$ are the {\it
pressures} at the vertices and the output of the plant. The $j^{\text{th}}$ element of the input to the controller which is given as in (\ref{P-connection}) is the pressure {\it difference}
$P_i - P_k$ across the $j^{\text{th}}$ edge.
The proportional part $\mu=-R\zeta=-RB^T\frac{\partial H}{\partial x}(x)$ of the PI controller (\ref{PI}) corresponds to adding {\it
damping} to the dynamics (proportional to
the pressure differences along the edges). The integral part of the controller
has the interpretation of adding {\it compressibility} to the compressible fluid
network dynamics. Using this emulated compressibility, the PI-controllers
(\ref{PI}) is able to regulate the fluid network to a output agreement situation where all pressures
$P_i$ are equal, irrespective of the constant inflow and outflow $\bar{d}$ satisfying
the matching condition (\ref{matching}). 
However, since the PI controller (\ref{PI}) can introduce the oscillation which can make the state variables $x$ of the closed-loop (\ref{closedloop}) become negative. This is clearly infeasible.
\end{exmp}

This example motivates us to modify the flow connections to time varying 
upper and lower bounds, i.e.{\small
\begin{equation}\label{F-connection-timevarying}
 u(t)=B\sat(\mu(t);\mu^-(t),\mu^+(t))
\end{equation}
}such that the states variables $x$ remain greater 
or equal than a lower bound for all time (for instance, zero in Example 
\ref{Hydraulic network}), while the outputs of the plant $\frac{\partial 
H}{\partial x_i}$ still converges to consensus.

For the rest of the paper we focus on the closed-loop system 
(\ref{plant},\ref{PI},\ref{P-connection},\ref{F-connection-timevarying}), given as {\small
\begin{equation}\label{closedloop_timevarying_basic}
\begin{array}{rcl}
\dot{x} & = & B \sat(-RB^T\frac{\partial H}{\partial
x}(x)-\frac{\partial H_c}{\partial
\eta}(\eta), \mu^-(t), \mu^+(t)), \\[2mm]
\dot{\eta} & = & B^T\frac{\partial H}{\partial x}(x).
\end{array}
\end{equation}
}with in/outflows are zero 
(i.e., $\bar{d}=0$). Furthermore, we assume 
$H_c(\eta)\in\mathbb{C}^1$, $H(x)=\sum_{i=1}^nH_i(x_i)\in\mathbb{C}^2$
with $H_i(\cdot):\mathbb{R}\rightarrow
\mathbb{R}$ be strictly convex and $\arg \min H_i(x_i)=\gamma_i,
i=1,2,\ldots,n$, $u^-(t)$ and $u^+(t)$ are parameters to be designed such that $x(t)\geq
\gamma, \forall t\geq 0$. Notice that this is equivalent to keeping the output of the plant $\frac{\partial H}{\partial x}$ being non-negative. 

\begin{remark}
Note that when $H(x)= H_1(x_1) + \ldots + H_n(x_n)$ and $H_i$ are convex, there are many controllers which fulfill the control aim of driving output of the plant $\frac{\partial H}{\partial x}(x)$ to agreement while keeping it non-negative. For example the {\it proportional} controller {\small
\begin{equation}\label{P}
\mu=-R\zeta
\end{equation}
}with $R$ a positive diagonal matrix has the property that the evolution of the closed-loop system (\ref{plant},\ref{F-connection},\ref{P-connection},\ref{P}) with $\bar{d}=0$, i.e.
{\small
\[
\dot{x} = -BRB^T\frac{\partial H}{\partial x}(x)
\]
}remains in the set $\{x\mid \frac{\partial H}{\partial x}(x(t))\in\mathbb{R}^n_+, \forall t\geq 0\}$ whenever $\frac{\partial H}{\partial x}(x(0))\geq 0$. This directly follows from
the properties of the weighted Laplacian matrix $BRB^T$: whenever at a certain
moment $\frac{\partial H_i}{\partial x_i}(x_i(t))=0$, then $\dot{x}_i(t)=-\sum_j r_k(\frac{\partial H_i}{\partial x_i}(x_i(t))-\frac{\partial H_j}{\partial x_j}(x_j(t)))\geq 0$ where $r_k$ is the $k^{\text{th}}$ diagonal element of $R$ and $e_k\sim (v_i,v_j)$. However for the second order system (\ref{closedloop}), in order to achieve the control aim the flows on the edges need to be regulated.
\end{remark}

\begin{exmp}[Hydraulic network continued]\label{Hydraulic network continued}
In this example we want to show that instead of keeping $\frac{\partial H_i}{\partial x_i}>0, i=1,2,\ldots,n$, they can be above any value $\bar{y}_i$ satisfying $\bar{y}_i=\frac{\partial H_i}{\partial x_i}(\bar{x}_i)$ for some $\bar{x}_i$ for all time. In that case we just replace the Hamiltonian $H$ in the system (\ref{closedloop_timevarying_basic}) by the Bregman distance with respect to $\bar{x}_i$, i.e. $H'_i(x_i)=H_i(x_i)-\bar{y}_i (x_i-\bar{x}_i)-H(\bar{x}_i)$, then the design of parameters $u^-(t)$ and $u^+(t)$ as done in the next section will keep $\frac{\partial H_i}{\partial x_i}>\bar{y}_i$. For instance with the same setting as in Example \ref{Hydraulic network}, suppose we want to keep the pressure of each reservoir greater or equal than
a given value $\bar{P}_i=\rho g \bar{h}_i\geq 0, i=1,2,\ldots,n$, then we can {\it modify} the  Hamiltonian as $H_i(x_i)=\frac{\rho g}{2S_i}(x_i-S_i \bar{h}_i)^2$, which is the {\it extra} stored energy in the
$i$th reservoir compared to the stored energy at height $\bar{h}_i$. In this case $\frac{\partial
H}{\partial x_i}=\rho g(h_i-\bar{h}_i)$ which is the relative pressure with respect to the height $\bar{h}_i$.
\end{exmp}

\section{The design of the flow constraints}\label{design}

In this section we will design the parameters $\mu^-(t)$ and $\mu^+(t)$ in 
(\ref{closedloop_timevarying_basic}).
The only situation in which a state variable $x_i$  may become smaller than 
$\gamma_i$ is that at a certain time instant $t$, $x_i(t)=\gamma_i$ and 
$\dot{x}_i(t)<0$. The basic idea underlying the design of the time-varying flow 
constraints is to eliminate this situation by adding saturation on the flows in 
the edges in such a way that $\dot{x}_i(t)\geq 0$ whenever $x_i(t)=\gamma_i$.
For each time $t$ and each vertex $v_i$, the edges adjacent to it can be divided 
into two sets{\small
\begin{equation}
 \begin{aligned}
  f_{v_i}^{in}(t) & = \{e_j \in \mathcal{E} \mid B_{ij}\mu_j>0 \} \\
f_{v_i}^{out}(t) & = \{e_j \in \mathcal{E} \mid B_{ij}\mu_j<0\}.
 \end{aligned}
\end{equation}
}For each time $t$, the vertices of the network can be divided into the 
following 
subsets, referred to as {\it white, gray and black} (with the last category 
divided into two subsets) {\small
\begin{equation}
 \begin{aligned}
  \mathcal{V}^W(t) & = \{v_i\in\mathcal{V} \mid x_i(t)>\gamma_i \} \\
\mathcal{V}^G(t) & = \{ v_i\in\mathcal{V} \mid x_i(t)=\gamma_i \} \\
\mathcal{V}^{B1}(t) & = \{v_i\in\mathcal{V}^G \mid B(i,:)\mu(t)<0 \} \\
\mathcal{V}^{B2}(t) & = \{v_i\in\mathcal{V}^G \mid \exists
v_j\in\mathcal{V}^{B1} \textrm{ s.t. } f_{v_i}^{in}(t)\cap f_{v_j}^{out}(t) \neq
\emptyset  \}
 \end{aligned}\nonumber
\end{equation}
}where $B(i,:)$ is the $i$th row of $B$. Furthermore, we denote
$\mathcal{V}^B(t)=\mathcal{V}^{B1}(t)\cup\mathcal{V}^{B2}(t)$.

\begin{exmp}\label{intuitive_fork}
Let us consider a part of the network given as given in Fig.\ref{structure1}. 
This example shows that the states of the \textit{black} nodes can become negative. 
Indeed suppose that at time $t$ the state variable at $v_2$, i.e. $x_2(t)$, 
decreases to $\gamma_2$, while $\mu_2(t)+\mu_3(t)>\mu_1(t)\geq 0$, then 
$\dot{x}_2(t)<0$.
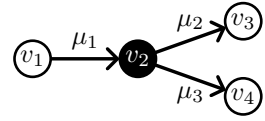
\begin{figure}[ht]
\begin{center}
\begin{tikzpicture}
\tikzstyle{EdgeStyle}    = [thin,double= black,
                            double distance = 0.5pt]
\useasboundingbox (0,0) rectangle (3cm,1.5cm);

\Vertex[style={minimum
size=1pt,shape=circle,inner sep=1pt},LabelOut=false,L=\hbox{$v_1$},x=0.1cm,y=0.8cm]{v1}
\Vertex[style={minimum
size=1pt,shape=circle,fill=black,text=white,inner sep=1pt},LabelOut=false,L=\hbox{$v_2$},x=1.5cm,y=0.8cm]{v2}
\Vertex[style={minimum
size=2pt,shape=circle,inner sep=1pt},LabelOut=false,L=\hbox{$v_3$},x=2.9cm,y=1.3cm]{v3}
\Vertex[style={minimum
size=2pt,shape=circle,inner sep=1pt},LabelOut=false,L=\hbox{$v_4$},x=2.9cm,y=0.3cm]{v4}
\draw
(v1) edge[->,>=angle 90,thin,double= black,double distance = 0.5pt]
node[above]{$\mu_1$} (v2)  
(v2) edge[->,>=angle 90,thin,double= black,double distance = 0.5pt]
node[above]{$\mu_2$} (v3)
(v2) edge[->,>=angle 90,thin,double= black,double distance = 0.5pt]
node[below]{$\mu_3$} (v4);
\end{tikzpicture}
\caption{Explanation about how the \textit{black} vertices may exhibit negative state values, with $\mu_i$ the output of the controller on the $i$-th edge.}\label{structure1}
\end{center}
\end{figure}
\end{exmp}

Let us denote the set of outgoing edges of all vertices in 
$\mathcal{V}^B(t)$, i.e., 
$\cup_{v_i\in\mathcal{V}^B(t)}f_{v_i}^{out}$, as $\mathcal{E}^B_{out}(t)$. Along
the edges $e_l\in\mathcal{E}^B_{out}(t)$, a saturation $[-|\phi_l^*(t)|, 
|\phi_l^*(t)|]$ is imposed on the flow, while along the rest of the edges
there are no saturations where $\phi^*(t)\in\mathbb{R}^m$ 
is the optimal solution of the following optimization problem{\small
\begin{equation}\label{optimization}
\begin{aligned}
  \min_{\phi}&
\sum_{e_j\in\mathcal{E}^B_{out}(t)}\frac{1}{2|\mu_j(t)|}
\Big(\big(\phi_j-\mu_j(t)\big)^2+\phi_j^2\Big)\\
\textrm{s.t. } & B(i,:)\phi=0,\forall v_i\in\mathcal{V}^B(t), \\
  &\phi_k=\mu_k(t), \textrm{ if } 
e_k\in\mathcal{E}\setminus\mathcal{E}^B_{out}(t).
\end{aligned}
\end{equation}
}
Furthermore, let us denote {\small
\begin{equation}\label{upperbound}
\phi^+_j (t) = \left\{ \begin{array}{ll}
|\phi^*_j(t)| & \textrm{if $e_j\in \mathcal{E}^B_{out}(t)$}\\
+\infty & \textrm{else},
\end{array}
\, j=1,\ldots,m. \right.
\end{equation}
}then the closed-loop (\ref{closedloop_timevarying_basic}) can be written as 
{\small
\begin{equation}\label{closedloop_timevary_main}
\begin{array}{rcl}
\dot{x} & = & B \sat(-RB^T\frac{\partial H}{\partial
x}(x)-\frac{\partial H_c}{\partial
\eta}(\eta), -\phi^+(t), \phi^+(t)), \\[2mm]
\dot{\eta} & = & B^T\frac{\partial H}{\partial x}(x)
\end{array}
\end{equation}
}with the parameters $\mu^-(t)$ and $\mu^+(t)$ in 
(\ref{closedloop_timevarying_basic}) being chosen as $-\phi^+(t)$ and $ 
\phi^+(t)$ respectively.

\begin{exmp}
Continuing Example \ref{intuitive_fork}, suppose at time $t$, the flows 
are subject to $\mu_2(t)+\mu_3(t)>\mu_1(t)\geq 0$ and $x_2(t)=\gamma_2$. Then
the above control protocol will set $ \dot{x}_2=\mu_1(t) + 
\sum_{i=2,3}\sat(\mu_i(t),-\phi^+_i(t),\phi^+_i(t))$
where
$\phi^+_i(t)=|\frac{\mu_1(t)\mu_i(t)}{\sum_{i=2,3}\mu_i(t)}|, i=2,3.$
\end{exmp}

Furthermore, the solution of the optimization problem (\ref{optimization}) 
can be seen as the limit of the following algorithm.

{\bf Algorithm}: {\it Initialization}: at time $t$ when there are grey nodes in the
network, set the initial value $\phi^0=\mu(t)\in\mathbb{R}^m$. 
\\
{\it Step} $k$: Let
$\phi^{k-1}$ be the value from the previous step $k-1$. If there exists a node 
$i$ such
that $B(i,:)\phi^{k-1}< 0$, then {\small
\begin{equation}
\phi^k_j  = \left\{ \begin{array}{ll}
\frac{\sum_{e_j\in f^{in}_{v_i}}|\phi^{k-1}_j|}{\sum_{e_j\in
f^{out}_{v_i}}|\phi^{k-1}_j|}\phi^{k-1}_j & \textrm{if $e_j\in f^{out}_{v_i}$}\\
\phi^{k-1}_j & \textrm{else},
\end{array}
\, j=1,\ldots,m. \right.
\end{equation}
}This algorithm is converging since in every iteration the absolute values of the
flows are non-increasing. It can be proved that $\lim_{k\rightarrow 
\infty}\phi^{k}(t)=\phi^*(t)$. However the proof is omitted due to lack of space.  

\begin{exmp}\label{intuitive_circle}
In this example, we consider the structure of the network as given in Fig.
\ref{structure2}. Suppose at time $t$, $x_1(t)>\gamma_1,
x_2(t)=\gamma_2, x_3(t)=\gamma_3$ and the output of controller (\ref{PI}) is
$\mu(t)=[1,3,1,2]^T$, then $\mathcal{V}^B(t)=\{v_2,v_3\}$. By 
using the algorithm, the flows on $f_2^{out}(t)\cup 
f_3^{out}(t)=\{e_2,e_3,e_4\}$, are
saturated to the values $\frac{3}{2},\frac{1}{2}$ and $1$, respectively. Furthermore, it can
be verified that $[\frac{3}{2},\frac{1}{2},1]^T$ is the solution of
optimization problem (\ref{optimization}).
 
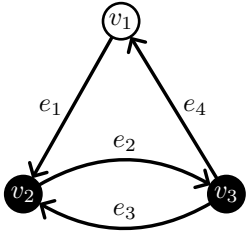
\begin{figure}[ht]
\begin{center}
\begin{tikzpicture}
\tikzstyle{EdgeStyle}    = [thin,double= black,
                            double distance = 0.5pt]
\useasboundingbox (0,0) rectangle (3cm,3cm);

\Vertex[style={minimum
size=0.15cm,shape=circle,fill=white,text=black, inner sep=1pt},LabelOut=false,L=\hbox{$v_1$}, x=1.5cm , y=2.8cm] {v1}
\Vertex[style={minimum
size=0.15cm,shape=circle,fill=black,text=white,inner sep=1pt},LabelOut=false,L=\hbox{$v_2$},x=0.2cm,y=0.5cm]{v2}
\Vertex[style={minimum
size=0.15cm,shape=circle,fill=black,text=white,inner sep=1pt},LabelOut=false,L=\hbox{$v_3$},x=2.9cm,y=0.5cm]{v3}
\draw
(v1) edge[->,>=angle 90,thin,double= black,double distance = 0.5pt]
node[left]{$e_1$} (v2)
(v2) edge[->,>=angle 90,bend left, thin,double= black,double distance = 0.5pt]
node[above]{$e_2$} (v3)
(v3) edge[->,>=angle 90,bend left, thin,double= black,double distance = 0.5pt]
node[above]{$e_3$} (v2)
(v3) edge[->,>=angle 90,thin,double= black,double distance = 0.5pt]
node[right]{$e_4$} (v1);
\end{tikzpicture}
\caption{Network of Example $\ref{intuitive_circle}$}\label{structure2}
\end{center}
\end{figure}
\end{exmp}


\section{Stability analysis}\label{proof}

In this section we will prove the stability of the system (\ref{closedloop_timevary_main}), and its convergence to consensus.

Since the right-hand-side of the system (\ref{closedloop_timevary_main}) is
discontinuous, we will consider Filippov solutions. The notations are taken
from \cite{cortes2008}.

\begin{definition}(\cite{cortes2008})
Let $\mathfrak{B}(\mathbb{R}^d)$ denote the collection of subsets of
$\mathbb{R}^d$. For $X:\mathbb{R}^d\rightarrow\mathbb{R}^d$, define the
\textit{Filippov set-valued map} $F[X]:\mathbb{R}^d\rightarrow
\mathfrak{B}(\mathbb{R}^d)$ as{\small
\begin{equation}
 F[X](x)\triangleq
\bigcap_{\delta>0}\bigcap_{\mu(S)=0}\overline{co}\{X(B(x,\delta)\backslash S)\}
\end{equation}}
\end{definition}

\begin{definition}
 A Filippov solution of $\dot{x}(t)=X(x(t))$ on $[0,t_1]\subset\mathbb{R}$ is
an absolutely continuous map $x:[0,t_1]\rightarrow\mathbb{R}^d$ that satisfies 
{\small
\begin{equation}
 \dot{x}(t)\in F[X](x)
\end{equation}
}for almost all $t\in[0,t_1]$.
\end{definition}

Here are two useful facts about computing the Filippov set-valued map.
\begin{proposition}\label{filippov properties}(\cite{cortes2008})
{\it Product Rule}: If $X_1:\mathbb{R}^d\rightarrow\mathbb{R}^d$ and
$X_2:\mathbb{R}^d\rightarrow\mathbb{R}^n$ are locally bounded at
$x\in\mathbb{R}^d$, then{\small
\begin{equation}
 F[(X_1,X_2)^T](x)\subseteq F[X_1](x)\times F[X_2](x).
\end{equation}
}Moreover, if either $X_1$ or $X_2$ is continuous at $x$, then equality holds.

{\it Matrix Transformation Rule}: If $X:\mathbb{R}^d\rightarrow\mathbb{R}^m$ is
locally bounded at $x\in\mathbb{R}^d$ and
$Z:\mathbb{R}^d\rightarrow\mathbb{R}^{d\times m}$ is continuous at
$x\in\mathbb{R}^d$, then {\small
\begin{equation}
 F[ZX](x)=Z(x)F[X](x).
\end{equation}}
\end{proposition}

\begin{theorem}
Consider the system (\ref{closedloop_timevary_main}) on the graph $\mathcal{G}$ in closed loop
with the saturation bounds as given in (\ref{upperbound}). Assume that
$H=\sum_{i=1}^n H_i(x_i)\in\mathbb{C}^2$ and $H_c\in\mathbb{C}^1$ are positive definitive and
radially unbounded. Furthermore $H_i$ are strictly convex with
$\arg\min_{x\in\mathbb{R}^n}H(x)=\gamma\in\mathbb{R}^n$. Then

(i) $x(t)\geqslant \gamma$ for all $t> 0$ if $x(0)\geq
\gamma$;

(ii) the trajectories of the closed-loop system
(\ref{closedloop_timevary_main}) will converge to an element of the load
balancing set{\small
\begin{equation}
\mathcal{E}_{\mathrm{tot}} = \{ (x,\eta) \mid \frac{\partial H}{\partial x}(x) =
\alpha \mathds{1}, \, \alpha \in \mathbb{R}^+, \, B\frac{\partial H_c}{\partial
\eta}(\eta)=0\, \}.\nonumber
\end{equation}
}if and only if $\mathcal{G}$ is weakly connected.
\end{theorem}

\begin{proof}
(i) It can be verified from the form of optimization problem
(\ref{optimization}) which grantee that $\dot{x}_i(t)\geq 0$ when
$x_i(t)=\gamma_i.$

(ii) \textit{Sufficiency}. 
First by using Proposition \ref{filippov properties}, the differential equations 
(\ref{closedloop_timevary_main}) are 
replaced by the differential inclusion{\small
\begin{equation}
\begin{aligned}
 \begin{bmatrix}
  \dot{x}(t) \\ \dot{\eta}(t)
 \end{bmatrix}
& \in F\Bigg[ 
\begin{bmatrix}
 B \sat(\mu(t),-\phi^+(t),\phi^+(t)) \\
B^T\frac{\partial H}{\partial x}(x)
\end{bmatrix}
\Bigg] \\
& = \begin{bmatrix}
     BF\Big[\sat(\mu(t),-\phi^+(t),\phi^+(t))\Big] \\
B^T\frac{\partial H}{\partial x}(x)
    \end{bmatrix}\\
&\triangleq F(x,\eta)
\end{aligned}
\end{equation}
}where the equality is implied by Proposition \ref{filippov properties}. Notice
that the set-valued map $F(x,\eta)$ is locally bounded and its values are nonempty, compact and convex
sets. Furthermore, for each $t\in\mathbb{R}, (x,\eta)\rightarrow F(x,\eta)$ is
upper semi-continuous.

Take as Lyapunov function the Hamiltonian function {\small
\begin{equation}
V(x,\eta) := H(x) + H_c(\eta),
\end{equation}
}which is differentiable. Then the set-valued Lie derivative 
$\tilde{\mathcal{L}}_F
V:\mathbb{R}^{n+m}\rightarrow\mathfrak{B}(\mathbb{R})$ of
$V$ with respect to $F$ at $(x,\eta)$ is defined as {\small
\begin{equation}
 \begin{aligned}
  \tilde{\mathcal{L}}_F V & = \{(\nabla V)^T\omega \mid
\omega\in F(x,\eta)\} \\
& =  \frac{\partial^T H}{\partial 
x}(x)BF\Big[\sat(\mu(t),-\phi^+(t),\phi^+(t))\Big] \\
& + \frac{\partial^T H}{\partial x}(x)B\frac{\partial
H_c}{\partial \eta}(\eta)
 \end{aligned}
\end{equation}
}For the $i$-th edge, the Filippov set-valued map is
given as {\small
\begin{equation}
\begin{aligned}
 & F\big[\sat(\mu_i(t),-\phi^+_i(t),\phi^+_i(t))\big] \\
 \subset & \left\{ \begin{array}{ll}
[0, \mu_i(t)] & e_i\in\mathcal{E}^B_{out}(t)\wedge\mu_i(t)>0,\\
\relax[\mu_i(t), 0\relax] & e_i\in\mathcal{E}^B_{out}(t)
\wedge\mu_i(t)<0,\\
\{\mu_i(t)\} & \textrm{else},
\end{array}
\, i=1,\ldots,m. \right.
\end{aligned}
\end{equation}
}

For the $i$-th edge of $\mathcal{G}$ on which $\phi^+_i(t)=+\infty$, i.e.
$e_i\in \mathcal{E}\setminus\mathcal{E}^B_{out}(t)$, we have {\small
\begin{equation}
\begin{aligned}
&\frac{\partial^T H}{\partial x}(x(t))B_i 
F\big[\sat(\mu_i(t),-\phi^+_i(t),\phi^+_i(t))\big]\\
& + \frac{\partial^T H}{\partial x}(x(t))B_i\frac{\partial H_{c}}{\partial
\eta_i}(\eta(t)) \\
= & -\frac{\partial^T H}{\partial x}(x(t))B_iB_i^T\frac{\partial H}{\partial x}(x(t))
\end{aligned}
\end{equation}
}where $B_i$ is the $i$-th column of $B$.

For the $i$-th edge on which $\phi^+_i(t)<+\infty$, i.e.
$e_i\in\mathcal{E}^B_{out}(t)$, we have that 
$\forall \omega_i\in F \big[\sat(-\mu_i(t),-\phi^+_i(t),\phi^+_i(t))\big]$; 
which can be
written as{\small
\begin{equation}
\begin{aligned}
 \omega_i = & (1-\kappa_i)0+\kappa_i (\mu_i(t)),\\
&\textrm{for some } \kappa_i \in  [0,1],
\end{aligned}
\end{equation}
}This implies that {\small
\begin{equation}
\begin{aligned}
&\frac{\partial^T H}{\partial x}(x)B_i F\big[ 
\sat(\mu_i(t),-\phi^+_i(t),\phi^+_i(t))\big]\\
& + \frac{\partial^T H}{\partial
x}(x)B_i\frac{\partial H_{c}}{\partial
\eta_i}(\eta) \\
= & \{-\kappa_i \frac{\partial^T H}{\partial x}(x)B_iB_i^T\frac{\partial
H}{\partial x}(x) \\
& +(1-\kappa_i)\frac{\partial^T H}{\partial
x}(x)B_i\frac{\partial H_{c}}{\partial
\eta_i}(\eta) \mid \kappa_i\in
[0,1] \}
\end{aligned}
\end{equation}
}Furthermore, when $\eta^+_i(t)<+\infty$, we have either

\textbullet \  $B_i^T\frac{\partial H}{\partial x}(x)\geqslant 0$ and
$-B_i^T\frac{\partial H}{\partial x}(x)-\frac{\partial H_c}{\partial
\eta_i}(\eta)>0$ which implies
$\frac{\partial^T H}{\partial
x}(x)B_i\frac{\partial H_c}{\partial \eta_{i}}(\eta)\leqslant-\frac{\partial^T H}{\partial
x}(x)B_iB_i^T\frac{\partial H}{\partial x}(x)$ or

\textbullet \  $B_i^T\frac{\partial H}{\partial x}(x)\leqslant 0$ and
$-B_i^T\frac{\partial H}{\partial x}(x)-\frac{\partial H_c}{\partial
\eta_{i}}(\eta)<0$ which implies
$\frac{\partial^T H}{\partial
x}(x)B_i\frac{\partial H_c}{\partial \eta_{i}}(\eta)\leqslant-\frac{\partial^T H}{\partial
x}(x)B_iB_i^T\frac{\partial H}{\partial x}(x)$ again.

So far, we can conclude that {\small
\begin{equation}
 \begin{aligned}
 & \frac{\partial^T
H}{\partial
x}(x)B_i\big( F\big[\sat(\mu_i(t),-\phi^+_i(t),\phi^+_i(t))\big]
+\frac{\partial H_{c}}{\partial
\eta_i}(\eta)\big) \\
\leqslant & -\frac{\partial^T H}{\partial
x}(x)B_iB_i^T\frac{\partial H}{\partial x}(x),
 \end{aligned}
\end{equation}
}i.e. $\max \tilde{\mathcal{L}}_F V_{\bar{d}}(x,\eta)\leqslant -\frac{\partial^T
H}{\partial x}(x)BB^T\frac{\partial H}{\partial x}(x)$.

By LaSalle's Invariance principle, the trajectories will converge to the largest invariant
set, denoted as $\mathcal{I}$, within the set where $\{(x,\eta)\mid \dot{V}=0\}$, i.e.
$\{(x,\eta)\mid B^T\frac{\partial H}{\partial x}(x)=0\}$. In
$\mathcal{I}$ we have {\small
\begin{equation}
B^T\frac{\partial^2 H}{\partial
x^2}B\sat(-\frac{\partial H_c}{\partial \eta}(\eta(t)), -\phi^+(t), \phi^+(t))=0 
\end{equation}
}which implies that $x$ remains at a constant value, denoted by $\nu$, in 
$\mathcal{I}$
and $\frac{\partial H}{\partial x}(\nu)=\alpha\mathds{1}.$ Furthermore, in view 
of $\nu\geq \gamma$ and the convexity of $H$ we can prove that $\alpha>0$. 
By the optimal control protocol given in the previous section, we have that all the vertices
will be white for large enough $t$, which implies at steady state
$B\frac{\partial H_c}{\partial
\eta}(\eta) = 0$. This concludes the proof.

\textit{Necessity}. If the graph is not weakly connected then the above
analysis will hold on every connected component, and the common value $\alpha$
will be different for different components. 
\end{proof}

\begin{exmp}[Hydraulic network continued]\label{Hydraulic network simulation}
In this example we show the simulation results of the hydraulic network defined on
the graph given in Figure \ref{network example}, with flow constraints given as solution
of the optimization problem (\ref{optimization}). The values of the parameters are taken as
$S_i=1m^2, i=1,\cdots,5, \rho=1 kg/m^3, \gamma=0$ and
$[x(0),x_c(0)]=[0,0.5,1,2,0,5,9,3,0,-1,-2,-4]$. In Figure
\ref{simulation1}, it can be seen that the volume of each reservoir
is kept nonnegative for all times. Furthermore the pressures of
reservoirs converge to a common value (consensus). 

\begin{figure}[ht]
\begin{center}
\begin{tikzpicture}
\tikzstyle{EdgeStyle}    = [thin,double= black,
                            double distance = 0.5pt]
\useasboundingbox (0,0) rectangle (4cm,3.5cm);
\tikzstyle{VertexStyle} = [shading         = ball,
                           ball color      = white!100!white,
                           minimum size = 20pt,%
                           inner sep       = 1pt,]
\Vertex[style={minimum
size=0.2cm,shape=circle},LabelOut=false,L=\hbox{$1$},x=1.5cm,y=3cm]{v1}
\Vertex[style={minimum
size=0.2cm,shape=circle},LabelOut=false,L=\hbox{$2$},x=0cm,y=1.5cm]{v2}
\Vertex[style={minimum
size=0.2cm,shape=circle},LabelOut=false,L=\hbox{$3$},x=1.5cm,y=0cm]{v3}
\Vertex[style={minimum
size=0.2cm,shape=circle},LabelOut=false,L=\hbox{$4$},x=3cm,y=1.5cm]{v4}
\Vertex[style={minimum
size=0.2cm,shape=circle},LabelOut=false,L=\hbox{$5$},x=0cm,y=3cm]{v5}
\draw
(v1) edge[->,>=angle 90,thin,double= black,double distance = 0.5pt]
node[right]{$e_1$} (v2)
(v2) edge[->,>=angle 90,thin,double= black,double distance = 0.5pt]
node[left]{$e_2$} (v3)
(v3) edge[->,>=angle 90,thin,double= black,double distance = 0.5pt]
node[right]{$e_3$} (v1)
(v1) edge[->,>=angle 90,thin,double= black,double distance = 0.5pt]
node[right]{$e_4$} (v4)
(v4) edge[->,>=angle 90,thin,double= black,double distance = 0.5pt]
node[right]{$e_5$} (v3)
(v1) edge[->,>=angle 90,thin,double= black,double distance = 0.5pt]
node[above]{$e_6$} (v5)
(v5) edge[->,>=angle 90,thin,double= black,double distance = 0.5pt]
node[left]{$e_7$} (v2);
\end{tikzpicture}
\caption{Network structure of Example
$\ref{Hydraulic network simulation}$}\label{network example}
\end{center}
\end{figure}
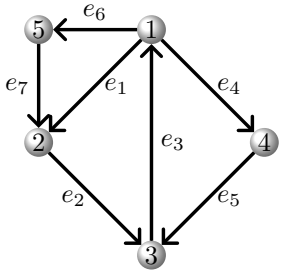

\begin{figure}[htbp]
 \centering
\includegraphics[width=0.45\textwidth]{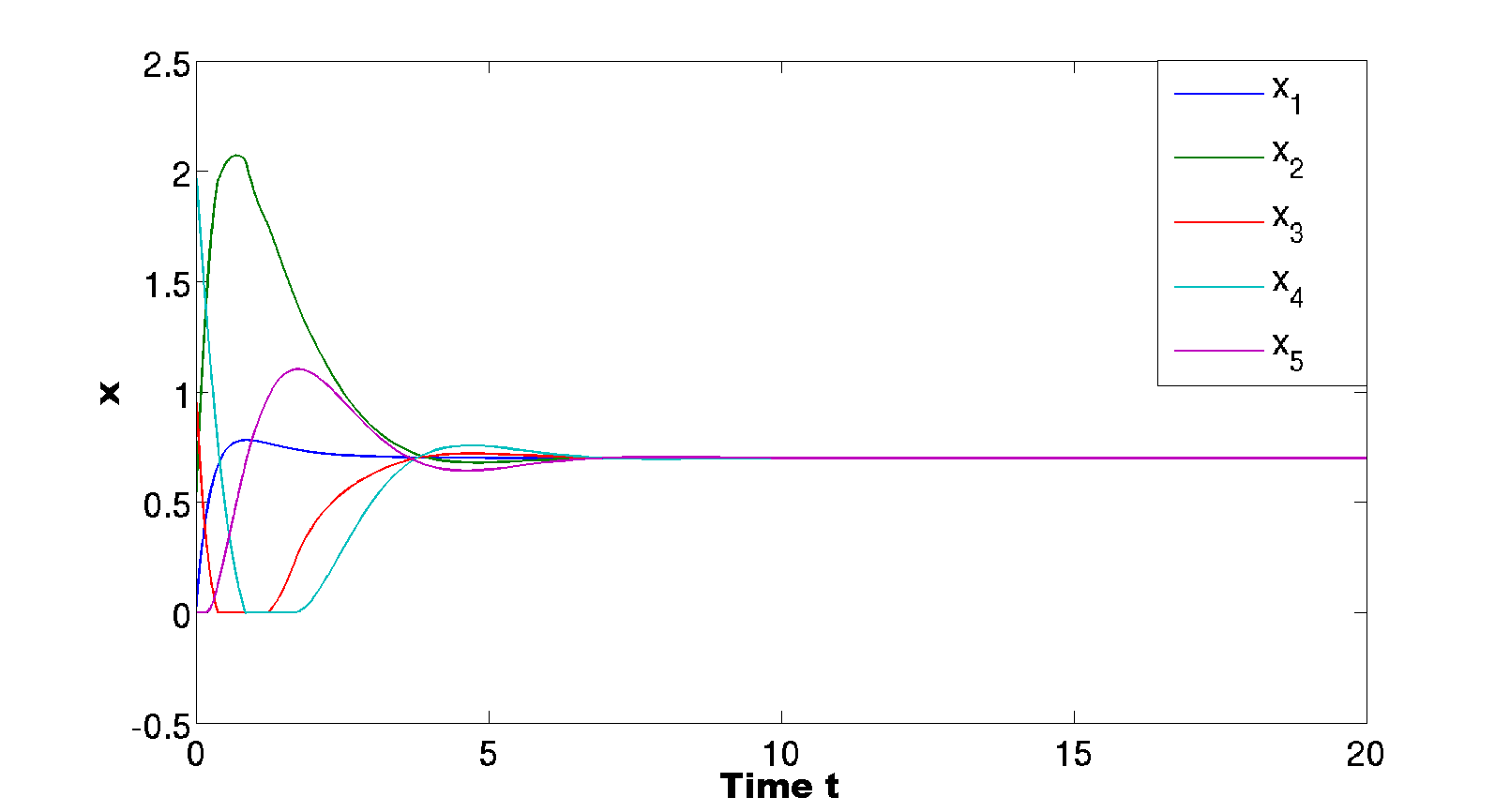}
 \caption{The time-evolutions $x_1(t),x_2(t),x_3(t),x_4(t),x_5(t)$ of the
system (\ref{closedloop_timevary_main}) defined on the graph as in Figure
\ref{network example} using the solution of (\ref{optimization}) as flow
constraints.}
 \label{simulation1}
 \end{figure}
\end{exmp}

\section{CONCLUSIONS}\label{conclusion}

We have considered a basic model of dynamical distribution networks with state
inequality constraints. We have formulated a distributed PI controller structure with time-varying
flow constraints which achieves consensus and maintains the state constraints. The flow constraints have been expressed
in terms of solutions of an optimization problem.   
We have discussed the existence of solutions for the system in the sense
of Filippov, and carried out the stability analysis of the network by
taking the Hamiltonian of the system as the Lyapunov function.

The results of this paper can be extended in a straightforward way to the case 
where the flows on the edges obey a priori constraints; for instance a limitation on the capacity of the pipes in hydraulic networks.




\bibliographystyle{IEEEtran} 
\bibliography{ifacconf}

\end{document}